\documentclass{article}
\usepackage{amsfonts}
\usepackage{amsmath}
\usepackage{float} 
\usepackage{booktabs}
\usepackage{float} 

\newtheorem{theorem}{Theorem}

\newtheorem{corollary}[theorem]{Corollary}

\newtheorem{example}[theorem]{Example}

\newtheorem{lemma}[theorem]{Lemma}

\newtheorem{proposition}[theorem]{Proposition}
\newtheorem{remark}[theorem]{Remark}

\newenvironment{proof}[1][Proof]{\noindent\textbf{#1.} }{\ \rule{0.5em}{0.5em}}

\begin{document}

\title{Harmonic numbers and Lerch functions behind falling factorial generating functions}
\author{Hemen Dutta, Gradimir V. Milovanovi\'c, Yilmaz Simsek 
\\Department of Mathematics,
	Gauhati University, \\Guwahati-781014, India\\hemen\_dutta08@rediffmail.com\\ Serbian Academy of Sciences and Arts \\ Kneza Mihaila 35, 11000 Beograd, Serbia $\&$ \\ University of Ni\v s, Faculty of Sciences and Mathematics,\\ 18000 Ni\v s, Serbia\\gvm@mi.sanu.ac.rs\\
Department of Mathematics, Faculty of Science,\\ University of Akdeniz,Antalya, Türkiye\\ ysimsek@akdeniz.edu.tr}

\maketitle

\begin{abstract}
We study a family of special numbers and polynomials introduced by Simsek
whose exponential generating function has as denominator a falling factorial
in the exponential variable. Using the partial fraction decomposition of
this denominator, each member of the family is expressed by a single finite
sum involving Bernoulli polynomials, Apostol--Bernoulli polynomials and
two--variable Fubini polynomials evaluated at reciprocals of integers. A
second expansion, around the point where the denominator vanishes, expresses
the same numbers through Stirling numbers of the second kind and generalized
harmonic numbers, and yields closed formulas for the first few members. We
construct a function, built from the Hurwitz zeta function and Lerch's
transcendent, which interpolates the polynomials at the negative integers and
gives a closed form of a multiple series recently introduced by Pokaz,
Ribi\v{c}i\'{c} Penava and Simsek; we give an Euler type integral
representation for it, and show that the numbers are signed combinations of
moments of geometric random variables. Recurrence
relations, asymptotics and an extension to arbitrary distinct nodes are also
given.
\textbf{Keywords.} Generating function, Bernoulli polynomials, Apostol--Bernoulli polynomials, Fubini polynomials, Stirling numbers, harmonic numbers, Lerch zeta function, geometric distribution

\textbf{MSC.2020} 05A15, 11B68, 11B73, 11M35, 33B15

\end{abstract}

\section{Introduction}

Generating functions under the Mellin transformation are commonly applied in theory of zeta functions and the Derichlet series and also approximation theory.  Generating functions have many useful
applications in various areas  including mathematics, probability, physics, etc. Our paper is normally motivated by the work of Simsek \cite{Simsek2021,SimsekSurvey}, which constructed exponential generating functions have as denominator a falling factorial in the exponential variable. By using these functions, Simsek defined the numbers $\ell
_{n}(\nu )$ and the polynomials 
$\ell_{n}(x;\nu )$. 

Recently, Pokaz, Ribi\v{c}i\'{c} Penava and Simsek \cite{Peneva1} gave the interpolation function for the polynomials $\ell_{n}(x;\nu )$. They showed that this function is a linear combination of the Hurwitz zeta function and the Lerch's transcendent. Moreover, they left open the problem of expressing this interpolation function in terms of these two functions and determining the necessary coefficients.

The goal of this paper is to not only derive many novel formulas and identities involving Bernoulli polynomials, Apostol--Bernoulli polynomials and
two--variable Fubini polynomials, the Harmonic numbers, the Stirling numbers of the second kind, the Bell numbers, the Frobenius--Euler polynomials etc., but also construct a function in terms of the Hurwitz zeta function and Lerch's transcendent, which interpolates the polynomials at the negative integers and gives a closed form of a multiple series recently given by Pokaz, Ribi\v{c}i\'{c} Penava and Simsek \cite{Peneva1} by using  generating functions and  partial fraction decomposition of this denominator methods. Another important motivation of this article is to not only extend the exponential generating functions and polynomials introduced by Simsek to arbitrary nodes, but also  derive new recurrence relations and asymptotic formulas.
\section{Preliminaries}
Throughout, $\mathbb{N}=\{1,2,3,\dots \}$, $\mathbb{N}_{0}=\mathbb{N}\cup
\{0\}$, and for $n\in \mathbb{N}_{0}$, $\nu \in \mathbb{N}_{0}$ we write 
\[
(n)_{\nu }=n(n-1)\cdots (n-\nu +1),\qquad (n)_{0}=1,
\]
for the falling factorial, so that $(n)_{\nu }=0$ when $\nu >n$ and 
\[
(n)_{\nu }=n!/(n-\nu)!
\]
otherwise. The same notation is used for a complex argument,
\[
(x)_{\nu }=x(x-1)\cdots (x-\nu +1).
\]

In \cite{Simsek2021} (see also \cite{SimsekSurvey}) the numbers $\ell
_{n}(\nu )=\ell _{n}((0,1,\dots ,\nu -1))$ and the polynomials 
$\ell_{n}(x;\nu )$ were defined, for $\nu \in \mathbb{N}$, by 
\begin{equation}\label{def-ell}
	G_{\nu }(w;x):=\frac{w^{\nu }e^{xw}}{\prod_{j=0}^{\nu -1}\left(
		e^{w}-j\right) }=\sum_{n=0}^{\infty }\ell _{n}(x;\nu )\frac{w^{n}}{n!},\qquad \ell _{n}(\nu ):=\ell _{n}(0;\nu ).  
\end{equation} 
In \cite{SimsekSurvey} the author observes that $G_{\nu }(w;x)$ (\ref{def-ell}) 
\[
\frac{w^{\nu }}{\prod_{j=0}^{\nu -1}(e^{w}-j)}=\frac{1}{(\nu -1)!}\cdot 
\frac{w}{e^{w}}\cdot \frac{w}{e^{w}-1}\cdot \frac{w}{\frac{1}{2}e^{w}-1} 
\cdots \frac{w}{\frac{1}{\nu -1}e^{w}-1},
\] 
i.e., the generating function is a product of the generating functions of
the Bernoulli numbers and of the Apostol--Bernoulli numbers $\mathcal{B}_{n}(\beta )$
with $\beta =1/j$, and remarks that further computational formulas can be
derived from this factorization. An explicit relation given there is the
one--step recursion 
\begin{equation}\label{simsek-raise}
	\ell _{n}(7)=\frac{1}{6}\sum_{k=0}^{n}\binom{n}{k}\mathcal{B}_{k}\!\left( \frac{1}{6}%
	\right) \ell _{n-k}(6),  
\end{equation}%
whose iteration produces a $(\nu -1)$--fold convolution, which is not
convenient for computation and does not display the structure of the numbers.

Simsek \cite{Simsek2021,SimsekSurvey} showed that special values of the
polynomials $\ell _{n}(x;\nu )$ reduce to the Frobenius--Euler polynomials,
the Apostol--Bernoulli polynomials, the Euler polynomials and the Bernoulli
polynomials. The polynomials $\ell _{n}(x;\nu )$ have been called the
\textit{Frobenius--Euler type Simsek polynomials} by Agyuz \cite{AgyuzFilomat};
they have been studied further by Rao et al.\ \cite{Rao,Rao2}, Kilar \cite%
{KilarMontes} and Ribi\v{c}i\'{c} Penava, Pokaz and Simsek \cite%
{Peneva,Peneva1}, and related families were investigated by Kucukoglu \cite%
{IK,KucukogluMTJPAM}.

The purpose of the present work is to replace the product structure by an
additive one. The starting point is the observation that the denominator in %
\eqref{def-ell} is a falling factorial in the variable $u=e^{w}$, 
\begin{equation*}
	\prod_{j=0}^{\nu -1}(e^{w}-j)=(e^{w})_{\nu },
\end{equation*}%
and that ${1}/{(u)_{\nu }}$ has a partial fraction decomposition with
simple poles at $u=0,1,\ldots ,\nu -1$. Each of the resulting pieces $
e^{xw}/(e^{w}-j)$ is a generating function of a classical family: for $j=1$
the Bernoulli polynomials, for $j\neq 1$ the Fubini (or geometric, or
ordered Bell) polynomials, in the form studied by Tanny \cite{Tanny},
Boyadzhiev \cite{Boyadzhiev} and Kilar and Simsek \cite{KilarSimsek2017}. One
thereby gets $\ell _{n}(x;\nu )$ as a single finite sum (Theorem \ref{thm:pf}%
). A second, independent expansion of ${1}/{(u)_{\nu }}$ around $u=1$
leads to Stirling numbers of the second kind and to the generalized harmonic
numbers 
\begin{equation*}
	H_{m}^{(r)}=\sum_{i=1}^{m}\frac{1}{i^{r}},\qquad H_{m}:=H_{m}^{(1)},\qquad
	H_{0}^{(r)}=0,
\end{equation*}%
and it is this second expansion which accounts for several closed forms, for
instance (Corollary \ref{cor:small}) 
\begin{equation*}
	\ell _{\nu }(\nu )=(-1)^{\nu }\nu (\nu -1)\left( H_{\nu -2}-\frac{3}{2}%
	\right) ,\qquad \ell _{\nu -1}(\nu )=(-1)^{\nu }(\nu -1).
\end{equation*}%
The partial fraction form also produces at once an interpolation function
for $\ell _{n}(x;\nu )$ (Theorem \ref{thm:interp}), in the sense in which
the Hurwitz zeta function interpolates the Bernoulli polynomials; it turns
out to be, up to an elementary factor, the interpolation function recently
introduced by Pokaz, Ribi\v{c}i\'{c} Penava and Simsek \cite{Peneva1}, whose
defining multiple series thereby collapses to a finite combination of Lerch
functions (Proposition \ref{prop:PRS}). The partial fractions also give a
probabilistic reading in terms of moments of geometric distributions
(Proposition \ref{prop:moments}), which connects the numbers $\ell _{n}(\nu
) $ with the recent work of B.~Simsek and coauthors \cite%
{BSimsek2024,SimsekKilar2025,SukruogluSimsek2025} on such moments.

Let us fix the remaining notation. The Stirling numbers of the second kind $%
S(n,k)$ and of the first kind $s(n,k)$ (signed) are defined by 
\begin{equation}
	\frac{(e^{w}-1)^{k}}{k!}=\sum_{n=0}^{\infty }S(n,k)\frac{w^{n}}{n!},\qquad
	(x)_{\nu }=\sum_{k=0}^{\nu }s(\nu ,k)x^{k}.  \label{stirling}
\end{equation}%
The Bernoulli polynomials $B_{n}(x)$ and the Apostol--Bernoulli polynomials $\mathcal{B}_{n}(x;\beta)$ (cf. \cite{Apostol1951,LuoSrivastava}) are given by 
\begin{equation}
	F_{a}(w,x,\beta )=\frac{we^{xw}}{\beta e^{w}-1}=\sum_{n=0}^{\infty }\mathcal{%
		B}_{n}(x;\beta )\frac{w^{n}}{n!},  \label{apostol}
\end{equation}%
where $B_{n}(x):=\mathcal{B}_{n}(x;1)$, $B_{n}:=B_{n}(0)$ and $\mathcal{B}%
_{n}(\beta ):=\mathcal{B}_{n}(0;\beta )$, so that $B_{1}=-1/2$ and, for $%
\beta \neq 1$, $\mathcal{B}_{0}(x;\beta )=0$, $\mathcal{B}_{1}(x;\beta
)=1/(\beta -1)$.
The Fubini polynomials $F_{n}(y)$ and the two--variable Fubini polynomials $%
F_{n}(x;y)$ are 
\begin{equation}
	\frac{e^{xw}}{1-y\left( e^{w}-1\right) }=\sum_{n=0}^{\infty }F_{n}(x;y)\frac{%
		w^{n}}{n!},\qquad F_{n}(y):=F_{n}(0;y)=\sum_{k=0}^{n}k!\,S(n,k)\,y^{k},
	\label{fubini}
\end{equation}%
and clearly%
\begin{equation*}
	F_{n}(x;y)=\sum_{k=0}^{n}\binom{n}{k}x^{n-k}F_{k}(y);
\end{equation*}%
$F_{n}(1)$ is the $n$-th Fubini (ordered Bell) number (cf. \cite{KilarSimsek2017,TKim2018,KucukogluMTJPAM}). We shall need the
values 
\begin{equation}
	F_{n}(x;-1)=(x-1)^{n},\qquad F_{n}(x;0)=x^{n},  \label{fubini-special}
\end{equation}%
the first because
\[1-(-1)(e^{w}-1)=e^{w}\] 
Finally,
\[
\Phi (z,s,x)=\sum_{r\geq 0}z^{r}(r+x)^{-s}
\]
is Lerch's transcendent, 
\[
\zeta (s,x)=\Phi (1,s,x)
\]
the Hurwitz zeta function, 
\[
Li_{s}(z)=z\Phi (z,s,1)
\]
the polylogarithm, and $Y_{m}(c_{1},\dots ,c_{m})$ the complete Bell
polynomials \cite[p.~134]{Comtet}, i.e., 
\begin{equation}
	\exp \bigg(\sum_{r\geq 1}c_{r}\frac{t^{r}}{r!}\bigg)=\sum_{m\geq
		0}Y_{m}(c_{1},\dots ,c_{m})\frac{t^{m}}{m!}.  \label{bell}
\end{equation}

Every identity between formal power series below is also an identity between
analytic functions in a neighbourhood of $w=0$, and we will not distinguish
between the two points of view except when convergence matters (Sections \ref%
{sec:interp} and \ref{sec:asymp}).

\section{Two expansions of $\dfrac{e^{xw}}{e^{w}-a}$}

\label{sec:lemmas}

The results of this paper depend on the following elementary lemma, which is
stated with its domain of convergence because this is needed for the
interpolation function in Section \ref{sec:interp}.

\begin{lemma}
	\label{lem:fubini} Let $a\in \mathbb{C}$, $a\neq 1$, $x\in \mathbb{C}$. Then 
	\begin{equation}
		\frac{e^{xw}}{e^{w}-a}=-\frac{1}{a-1}\sum_{m=0}^{\infty }F_{m}\left( x;\frac{%
			1}{a-1}\right) \frac{w^{m}}{m!},  \label{eq:lem-fubini}
	\end{equation}%
	as formal power series in $w$, and as analytic functions for $|w|<\log
	(1+|a-1|)$.
\end{lemma}

\begin{proof}
	Put $u=e^{w}-1$, so $e^{w}-a=u-(a-1)=-(a-1)\big(1-u/(a-1)\big)$. Since $u\in
	w\,\mathbb{C}[[w]]$, the geometric series 
	\begin{equation*}
		\frac{1}{1-u/(a-1)}=\sum_{k\geq 0}\frac{u^{k}}{(a-1)^{k}}
	\end{equation*}%
	is a well defined element of $\mathbb{C}[[w]]$ and by \eqref{fubini} (with $%
	y=1/(a-1)$) the product with $e^{xw}$ is $\sum_{m}F_{m}(x;1/(a-1))w^{m}/m!$.
	For the analytic statement note $|e^{w}-1|\leq e^{|w|}-1<|a-1|$ whenever $%
	|w|<\log (1+|a-1|)$.
\end{proof}

In \cite{F1} (see also \cite[Eq.~(37)]{FPTA}) the polynomials $Y_{n}(x,u;a)$
are defined, for $a>0$ and $u\neq 1$, by%
\begin{equation*}
	\frac{a^{xt}}{a^{t}-u}=\sum_{n=0}^{\infty }Y_{n}(x,u;a)\frac{t^{n}}{n!}.
\end{equation*}%
We use them with the base $a=e$ and write the parameter as $a$, so that%
\begin{equation}
	Q(x,w,a):=\frac{e^{xw}}{e^{w}-a}=\sum_{n=0}^{\infty }Y_{n}(x,a;e)\frac{w^{n}%
	}{n!}\qquad (a\neq 1).  \label{ye1}
\end{equation}%
Since the Frobenius--Euler polynomials $H_{n}(x;a)$ are generated by%
\begin{equation*}
	\frac{(1-a)e^{xw}}{e^{w}-a}=\sum_{n=0}^{\infty }H_{n}(x;a)\frac{w^{n}}{n!},
\end{equation*}%
one has $Y_{n}(x,a;e)=H_{n}(x;a)/(1-a)$; in particular $Y_{n}(x,-1;e)=\frac{1}{2}%
H_{n}(x;-1)=\frac{1}{2}E_{n}(x)$, where $E_{n}(x)$ denotes the Euler
polynomials. Comparing \eqref{ye1} with \eqref{eq:lem-fubini} gives%
\begin{equation}
	Y_{n}(x,a;e)=-\frac{1}{a-1}F_{n}\left( x;\frac{1}{a-1}\right) ,\qquad
	H_{n}(x;a)=F_{n}\left( x;\frac{1}{a-1}\right) \qquad (a\neq 1),
	\label{eq:FE-Fubini}
\end{equation}%
so that the Frobenius--Euler polynomials are the two--variable Fubini
polynomials at $y=1/(a-1)$; for $a=-1$ this is the relation $%
E_{n}(x)=F_{n}(x;-{1}/{2})$ which is used in Example \ref{ex:nodes}(b)
below.

\begin{lemma}
	\label{lem:apostol} Let $a\in \mathbb{C}$, $a\neq 0$. Then 
	\begin{equation}
		\frac{e^{xw}}{e^{w}-a}=\frac{1}{a}\sum_{m=0}^{\infty }\frac{\mathcal{B}%
			_{m+1}\left( x;{1}/{a}\right) }{m+1}\,\frac{w^{m}}{m!}\quad (a\neq
		1),\qquad \frac{e^{xw}}{e^{w}-1}=\sum_{m=0}^{\infty }\frac{B_{m+1}(x)}{m+1}\,%
		\frac{w^{m}}{m!}+\frac{1}{w}.  \label{eq:lem-apostol}
	\end{equation}%
	Consequently, for $a\notin \{0,1\}$ and $m\in \mathbb{N}_{0}$, 
	\begin{equation}
		\frac{\mathcal{B}_{m+1}\left( x;{1}/{a}\right) }{m+1}=-\frac{a}{a-1}%
		\,F_{m}\left( x;\frac{1}{a-1}\right) .  \label{eq:AB-F}
	\end{equation}
\end{lemma}

\begin{proof}
	By \eqref{apostol} with $\beta =1/a$, 
	\begin{equation*}
		\frac{e^{xw}}{e^{w}-a}=\frac{1}{a}\cdot \frac{1}{w}\cdot \frac{we^{xw}}{%
			a^{-1}e^{w}-1}=\frac{1}{a}\sum_{m\geq 0}\mathcal{B}_{m}\left( x;\frac{1}{a}\right) 
		\frac{w^{m-1}}{m!}.
	\end{equation*}%
	For $a\neq 1$ the term $m=0$ vanishes and reindexing gives the first
	formula; for $a=1$ the term $m=0$ is ${B_{0}(x)}/{w}={1}/{w}$.
	Comparing with \eqref{eq:lem-fubini} gives \eqref{eq:AB-F}.
\end{proof}

Comparing \eqref{ye1} with \eqref{eq:lem-apostol} also gives%
\begin{equation*}
	Y_{m}(x,a;e)=\frac{\mathcal{B}_{m+1}\left( x; {1}/{a}\right) }{a(m+1)}%
	,\qquad \mathcal{B}_{m+1}\left( x;\frac{1}{a}\right) =\frac{a(m+1)}{1-a}%
	H_{m}(x;a)\qquad (a\notin \{0,1\}),
\end{equation*}%
the second of which is the classical relation between the Apostol--Bernoulli
and the Frobenius--Euler polynomials, cf.\ \cite{F1}.

Relation \eqref{eq:AB-F} is known; for the two--variable Fubini polynomials
it is given in \cite{KilarSimsek2017}, and the polynomials $F_{n}(y)$
together with their generating function are studied in \cite[Section 3]%
{Boyadzhiev}. Since it is used repeatedly below, a short proof has been
included. Note also that \eqref{eq:lem-apostol} is meaningless at $a=0$,
while \eqref{eq:lem-fubini} at $a=0$ gives, by \eqref{fubini-special}, the
correct expansion of ${e^{xw}}/{e^{w}}=e^{(x-1)w}$; the Fubini form is
thus the more uniform one and we shall regard $j=0$ as an ordinary node in
what follows.

\begin{lemma}
	\label{lem:pf} For $\nu \in \mathbb{N}$ and $u\notin \{0,1,\dots ,\nu -1\}$, 
	\begin{equation}
		\frac{1}{(u)_{\nu }}=\sum_{j=0}^{\nu -1}\frac{A_{j}^{(\nu )}}{u-j},\quad
		A_{j}^{(\nu )}:=\prod_{\substack{ i=0  \\ i\neq j}}^{\nu -1}\frac{1}{j-i}=%
		\frac{(-1)^{\nu -1-j}}{j!\,(\nu -1-j)!}=\frac{(-1)^{\nu -1-j}}{(\nu -1)!}%
		\binom{\nu -1}{j}.  \label{eq:pf}
	\end{equation}
\end{lemma}

\begin{proof}
	The poles are simple and%
	\begin{equation*}
		A_{j}^{(\nu )}=\lim_{u\rightarrow j}(u-j)/(u)_{\nu };
	\end{equation*}%
	the product $\prod_{i\neq j}(j-i)$ equals%
	\begin{equation*}
		j!(-1)(-2)\cdots (-(\nu -1-j))=(-1)^{\nu -1-j}j!(\nu -1-j)!.
	\end{equation*}
\end{proof}

When $\nu $ is fixed we write $A_{j}$ for $A_{j}^{(\nu )}$. We record the
values%
\begin{equation*}
	A_{0}=(-1)^{\nu -1}/(\nu -1)!
\end{equation*}%
and%
\begin{equation*}
	A_{1}=(-1)^{\nu }/(\nu -2)!
\end{equation*}%
($\nu \geq 2$), and the identity%
\begin{equation*}
	\sum_{j=0}^{\nu -1}A_{j}=0
\end{equation*}%
for $\nu \geq 2$ (compare the coefficients of $u^{\nu -1}$ in $%
1=\sum_{j}A_{j}\prod_{i\neq j}(u-i)$).

\begin{remark}
	\label{rem:gkp} Lemma \ref{lem:pf} can also be read off from the identity %
	\cite[p.~188, Eq.~(5.41)]{Ronald}%
	\begin{equation}
		\sum_{j=0}^{m}\binom{m}{j}\frac{(-1)^{j}}{\alpha +j}=\frac{m!}{\alpha
			(\alpha +1)\cdots (\alpha +m)}=\frac{1}{\alpha \binom{m+\alpha }{m}},\qquad
		\alpha \notin \{0,-1,\ldots ,-m\},  \label{r1}
	\end{equation}%
	where $\binom{\alpha }{m}=(\alpha )_{m}/m!$. With $\alpha =-u$ and $m=\nu -1$
	it becomes%
	\begin{equation}
		\sum_{j=0}^{\nu -1}\binom{\nu -1}{j}\frac{(-1)^{j+1}}{u-j}=\frac{(\nu -1)!}{%
			-u(1-u)\cdots (\nu -1-u)}=\frac{(-1)^{\nu }(\nu -1)!}{(u)_{\nu }},
		\label{r3}
	\end{equation}%
	which is \eqref{eq:pf} with the values of $A_{j}^{(\nu )}$ inserted. For $%
	u=e^{w}$, \eqref{r3} says that%
	\begin{equation}
		\frac{(-1)^{\nu }(\nu -1)!}{w^{\nu }}\,G_{\nu }(w;0)=-e^{-w}+\frac{\nu -1}{%
			e^{w}-1}+\sum_{j=2}^{\nu -1}\binom{\nu -1}{j}(-1)^{j+1}Q(0,w,j),  \label{r4}
	\end{equation}%
	i.e. the generating function \eqref{def-ell} is, up to the factor $w^{\nu }$,
	a signed binomial combination of the functions $Q(0,w,j)$ of \eqref{ye1}
	and of $1/(e^{w}-1)$. The identities between the coefficients which follow
	from \eqref{r4} are those of the next section; see Corollary \ref{cor:FE}.
\end{remark}

\section{The partial fraction formula}
In this section, by using partial fraction formula, we derive novel identities involving the polynomials $\ell _{n}(x;\nu
)$, the Bernoulli polynomials, the Frobenius--Euler numbers, and the Stirling numbers of the second kind.

\label{sec:pf}

\begin{theorem}
	\label{thm:pf} Let $\nu \geq 2$ and $x\in \mathbb{C}$. Then $\ell _{n}(x;\nu
	)=0$ for $0\leq n\leq \nu -2$, $\ell _{\nu -1}(x;\nu )=(-1)^{\nu }(\nu -1)$,
	and for $n\geq \nu $ 
	\begin{equation}
		\ell _{n}(x;\nu )=\frac{(-1)^{\nu }(n)_{\nu -1}}{(\nu -2)!}B_{n-\nu
			+1}(x)-(n)_{\nu }\sum_{\substack{ j=0  \\ j\neq 1}}^{\nu -1}\frac{A_{j}}{j-1}%
		F_{n-\nu }\left( x;\frac{1}{j-1}\right) .  \label{eq:main-pf}
	\end{equation}%
	For $\nu =1$ one has%
	\begin{equation*}
		\ell _{n}(x;1)=n(x-1)^{n-1}.
	\end{equation*}
\end{theorem}

\begin{proof}
	The case $\nu=1$ is $we^{(x-1)w}$. Let $\nu\ge2$. By Lemma \ref{lem:pf} with 
	$u=e^{w}$, 
	\begin{equation}  \label{eq:G-split}
		G_\nu(w;x)=w^{\nu-1}A_1\frac{we^{xw}}{e^{w}-1}+w^{\nu}\sum_{j\ne1}A_j\frac{%
			e^{xw}}{e^{w}-j}.
	\end{equation}
	The first term equals $A_1\sum_{m\ge0}B_m(x)w^{m+\nu-1}/m!$, whose
	coefficient of $w^n/n!$ is 
	\begin{equation*}
		\frac{A_1\,n!}{(n-\nu+1)!}\,B_{n-\nu+1}(x)=A_1(n)_{\nu-1}B_{n-\nu+1}(x)\quad%
		\mbox{for $n\ge \nu-1$}
	\end{equation*}
	and zero for $n\le \nu-2$. By Lemma \ref{lem:fubini} (every $j\ne1$ is
	admissible) the $j$-th term of the sum is 
	\begin{equation*}
		-\frac{A_j}{j-1}\sum_{m\ge0}F_m\Bigl(x;\frac1{j-1}\Bigr)\frac{w^{m+\nu}}{m!},
	\end{equation*}
	with coefficient of $w^n/n!$ equal to 
	\begin{equation*}
		-\frac{A_j}{j-1}(n)_{\nu}F_{n-\nu}(x;\frac1{j-1})\quad\mbox{for $n\ge \nu$}
	\end{equation*}
	and zero otherwise. Since $A_1(\nu-1)_{\nu-1}B_0(x)=A_1(\nu-1)!=(-1)^\nu(%
	\nu-1)$, all assertions follow.
\end{proof}

Since the $j=0$ term of \eqref{eq:main-pf} is $A_0(n)_{\nu}(x-1)^{n-\nu}$ by %
\eqref{fubini-special}, and since $F_m(0;y)=F_m(y)$, we may write the
theorem at $x=0$ entirely with Stirling numbers of the second kind:

\begin{corollary}
	\label{cor:stirling} For $n\ge \nu\ge3$, 
	\begin{align}  \label{eq:cor-stirling}
		\ell_n(\nu)=\frac{(-1)^{n-1}(n)_{\nu}}{(\nu-1)!}&+\frac{(-1)^{\nu}(n)_{\nu-1}%
		}{(\nu-2)!}B_{n-\nu+1} \\
		&-(n)_{\nu}\sum_{j=2}^{\nu-1}\frac{(-1)^{\nu-1-j}}{j!\,(\nu-1-j)!\,(j-1)}%
		\sum_{k=0}^{n-\nu}\frac{k!\,S(n-\nu,k)}{(j-1)^{k}}.  \notag
	\end{align}
\end{corollary}

For $\nu =2$ the last sum is empty and%
\begin{equation*}
	\ell _{n}(2)=(-1)^{n-1}(n)_{2}+nB_{n-1};
\end{equation*}%
for instance $\ell _{3}(2)=6+3\cdot (1/6)={13}/{2}$. For $\nu =3$, using $%
\sum_{k}k!S(m,k)=F_{m}(1)$, 
\[
\ell _{n}(3)=\frac{1}{2}(-1)^{n-1}(n)_{3}-(n)_{2}B_{n-2}-\frac{1}{2}(n)_{3}F_{n-3}(1),
\]
which for $n=4$ gives $-12-12\cdot (1/6)-12\cdot 1=-26$ (the Fubini numbers
begin $1,1,3,13,75,\dots $). Table \ref{tab:values} was computed from %
\eqref{eq:cor-stirling} and checked against a direct expansion of %
\eqref{def-ell}.

\begin{table}[H]
	\caption{The numbers $\ell_n(\protect\nu)$, $1\le \protect\nu\le 6$, $0\le
		n\le 9$.}
	\label{tab:values}\centering
	\begin{tabular}{c|rrrrrrrrrr}
		\toprule $\nu\backslash n$ & 0 & 1 & 2 & 3 & 4 & 5 & 6 & 7 & 8 & 9 \\ 
		\midrule 1 & 0 & 1 & $-2$ & 3 & $-4$ & 5 & $-6$ & 7 & $-8$ & 9 \\ 
		2 & 0 & 1 & $-3$ & $\frac{13}{2}$ & $-12$ & $\frac{119}{6}$ & $-30$ & $\frac{%
			253}{6}$ & $-56$ & $\frac{717}{10}$ \\ 
		3 & 0 & 0 & $-2$ & 3 & $-26$ & $-60$ & $-839$ & $-7770$ & $-\frac{273172}{3}$
		& $-1179864$ \\ 
		4 & 0 & 0 & 0 & 3 & 0 & 80 & 450 & 5404 & 61320 & 807036 \\ 
		5 & 0 & 0 & 0 & 0 & $-4$ & $-\frac{20}{3}$ & $-\frac{580}{3}$ & $-\frac{14350%
		}{9}$ & $-\frac{558488}{27}$ & $-\frac{2439080}{9}$ \\ 
		6 & 0 & 0 & 0 & 0 & 0 & 5 & $\frac{35}{2}$ & $\frac{9485}{24}$ & $\frac{%
			150185}{36}$ & $\frac{2880437}{48}$ \\ 
		\bottomrule
	\end{tabular}%
\end{table}

Using Lemma \ref{lem:apostol} instead of Lemma \ref{lem:fubini} on the nodes 
$j\ge2$ we obtain the form which is closest to the product observed in \cite%
{SimsekSurvey}, but now as a sum.

\begin{corollary}
	\label{cor:apostol} For $n\geq \nu \geq 2$, 
	\begin{equation}
		\ell _{n}(x;\nu )=\frac{(-1)^{\nu -1}(n)_{\nu }}{(\nu -1)!}(x-1)^{n-\nu }+%
		\frac{(n)_{\nu -1}}{(\nu -1)!}\sum_{j=1}^{\nu -1}\binom{\nu -1}{j}\frac{(-1)^{\nu
				-1-j}\mathcal{B}_{n-\nu +1}\left( x;\frac{1}{j}\right) }{j},
		\label{eq:cor-apostol}
	\end{equation}%
	where%
	\begin{equation*}
		\mathcal{B}_{m}(x;1)=B_{m}(x).
	\end{equation*}
\end{corollary}

\begin{proof}
	By \eqref{eq:lem-apostol}, for $j\geq 2$ the coefficient of $w^{n}/n!$ in $%
	w^{\nu }A_{j}e^{xw}/(e^{w}-j)$ is 
	\begin{equation*}
		\frac{A_{j}}{j}(n)_{\nu }\frac{\mathcal{B}_{n-\nu +1}(x;1/j)}{n-\nu +1}=\frac{A_{j}}{j}%
		(n)_{\nu -1}\mathcal{B}_{n-\nu +1}(x;1/j),
	\end{equation*}%
	which has the same shape as the $j=1$ term.
\end{proof}

In terms of the polynomials $Y_{n}(x,a;e)$ of \eqref{ye1}, i.e. of the
Frobenius--Euler polynomials, Theorem \ref{thm:pf} reads as follows.

\begin{corollary}
	\label{cor:FE} For $n\geq \nu \geq 2$ and $x\in \mathbb{C}$,%
	\begin{align}
		\ell _{n}(x;\nu )& =\frac{(-1)^{\nu }(n)_{\nu -1}}{(\nu -2)!}B_{n-\nu
			+1}(x)+(n)_{\nu }\sum_{\substack{ j=0  \\ j\neq 1}}^{\nu -1}A_{j}\,Y_{n-\nu
		}(x,j;e)  \label{eq:cor-FE} \\
		& =\frac{(-1)^{\nu }(n)_{\nu -1}}{(\nu -2)!}B_{n-\nu +1}(x)-(n)_{\nu }\sum_{%
			\substack{ j=0  \\ j\neq 1}}^{\nu -1}\frac{A_{j}}{j-1}H_{n-\nu }(x;j),
		\notag
	\end{align}%
	where $H_{m}(x;0)=(x-1)^{m}$. At $x=0$ the last sum involves the
	Frobenius--Euler numbers $H_{m}(j):=H_{m}(0;j)$ at the integers $j=2,\dots
	,\nu -1$.
\end{corollary}

\begin{proof}
	By \eqref{ye1}, $w^{\nu }A_{j}Q(x,w,j)=A_{j}\sum_{m\geq 0}Y_{m}(x,j;e)w^{m+\nu
	}/m!$ for $j\neq 1$, so the first formula is \eqref{eq:G-split} read
	coefficientwise (for $j=0$, $Q(x,w,0)=e^{(x-1)w}$ and $%
	Y_{m}(x,0;e)=(x-1)^{m}=H_{m}(x;0)$). The second follows from $%
	Y_{m}(x,j;e)=H_{m}(x;j)/(1-j)$, or from \eqref{eq:main-pf} and %
	\eqref{eq:FE-Fubini}.
\end{proof}

\begin{remark}
	\label{rem:raise} Formula \eqref{simsek-raise}, taken from \cite%
	{SimsekSurvey}, is the case $\nu =6$ of the general \textquotedblleft
	raising\textquotedblright\ identity 
	\begin{equation}
		\ell _{n}(x;\nu +1)=\frac{1}{\nu }\sum_{k=0}^{n}\binom{n}{k}\mathcal{B}_{k}\!\left( 
		\tfrac{1}{\nu }\right) \ell _{n-k}(x;\nu ),  \label{eq:raise}
	\end{equation}%
	which is equivalent to 
	\begin{equation*}
		G_{\nu +1}(w;x)=G_{\nu }(w;x)\cdot \frac{1}{\nu }\cdot \frac{w}{\nu
			^{-1}e^{w}-1}.
	\end{equation*}%
	Starting from $\ell _{n}(x;1)=n(x-1)^{n-1}$ and iterating \eqref{eq:raise} $%
	\nu -1$ times one gets a $(\nu -1)$--fold convolution; \eqref{eq:cor-apostol}
	shows that this convolution reduces to a single sum, at the cost that the
	Apostol--Bernoulli polynomials appear at the index $n-\nu +1$ rather than at
	all intermediate indices. This is the usual advantage of partial fractions
	over products; it does not appear to have been applied to \eqref{def-ell}
	previously.
\end{remark}

\section{Expansion at $e^{w}=1$: Stirling and harmonic numbers}

\label{sec:harmonic}

Theorem \ref{thm:pf} treats all nodes alike. There is however a
distinguished node, $j=1$, because it is the only one at which the
denominator of \eqref{def-ell} vanishes for $w=0$. It is natural to separate
it and to expand the remaining rational function around $u=1$. Put, for $%
\nu\ge2$, 
\begin{equation}  \label{def-f}
	f_\nu(u):=\frac{u-1}{(u)_{\nu}}=\frac{1}{u\prod\limits_{j=2}^{\nu-1}(u-j)}\qquad%
	\Big(f_2(u)=\frac1u\Big),
\end{equation}
so that 
\begin{equation*}
	G_\nu(w;x)=w^{\nu-1}\cdot\frac{we^{xw}}{e^{w}-1}\cdot f_\nu(e^{w}).
\end{equation*}
The function $f_\nu$ is analytic in the disc $|u-1|<1$ (its poles are at $%
0,2,\dots,\nu-1$) and its Taylor coefficients there are computed in the next
lemma.

\begin{lemma}
	\label{lem:fder} For $\nu\ge2$ and $m\in\mathbb{N}_0$, 
	\begin{equation}  \label{eq:fder}
		f_\nu^{(m)}(1)=\frac{(-1)^{\nu}}{(\nu-2)!}\,Y_m\big(c_1,c_2,\dots,c_m\big)%
		,\qquad c_r:=(r-1)!\left((-1)^{r}+H^{(r)}_{\nu-2}\right).
	\end{equation}
	In particular, 
	\begin{equation*}
		f_\nu(1)=A_1=\frac{(-1)^\nu}{(\nu-2)!},\quad f_\nu^{\prime
		}(1)=A_1(H_{\nu-2}-1), \quad f_\nu^{\prime \prime }(1)=A_1\big(%
		(H_{\nu-2}-1)^2+1+H^{(2)}_{\nu-2}\big).
	\end{equation*}
\end{lemma}

\begin{proof}
	Write $f_\nu=e^{g}$ with $g=\log f_\nu$ on a neighbourhood of $u=1$ (any
	branch). Then 
	\begin{equation*}
		g^{\prime }(u)=-\frac1u-\sum_{j=2}^{\nu-1}\frac1{u-j},
	\end{equation*}
	hence for $r\ge1$ 
	\begin{equation*}
		g^{(r)}(u)=(-1)^{r}(r-1)!\biggl(u^{-r}+\sum_{j=2}^{\nu-1}(u-j)^{-r}\biggr)%
		,\quad g^{(r)}(1)=(r-1)!\biggl((-1)^{r}+\sum_{i=1}^{\nu-2}i^{-r}\biggr)=c_r .
	\end{equation*}
	By the exponential formula \eqref{bell} (equivalently, Fa\`a di Bruno for $%
	\exp\circ g$), 
	\begin{equation*}
		f_\nu^{(m)}=f_\nu\cdot Y_m(g^{\prime (m)})\quad\mbox{and}\quad f_\nu(1)=%
		\frac{1}{\prod\limits_{j=2}^{\nu-1}(1-j)}=\frac{(-1)^{\nu-2}}{(\nu-2)!}.
	\end{equation*}
	The special cases follow from $Y_1(c_1)=c_1$, $Y_2(c_1,c_2)=c_1^2+c_2$.
\end{proof}

\begin{theorem}
	\label{thm:taylor} Let $\nu\ge2$, $x\in\mathbb{C}$ and $n\ge \nu-1$. Then 
	\begin{equation}  \label{eq:taylor}
		\ell_n(x;\nu)=\frac{(-1)^{\nu}}{(\nu-2)!}(n)_{\nu-1}B_{n-\nu+1}(x)
		+(n)_{\nu}\sum_{m=1}^{n-\nu+1}\frac{f_\nu^{(m)}(1)}{m}\sum_{i=m-1}^{n-\nu}%
		\binom{n-\nu}{i}S(i,m-1)\,x^{\,n-\nu-i},
	\end{equation}
	with $f_\nu^{(m)}(1)$ given by \eqref{eq:fder}.
\end{theorem}

\begin{proof}
	For $|w|<\log2$ we have $|e^{w}-1|<1$, so Taylor's formula for $f_\nu$ at $%
	u=1$ gives 
	\begin{align*}
		G_\nu(w;x)&=w^{\nu-1}\frac{we^{xw}}{e^{w}-1}\sum_{m\ge0}\frac{f_\nu^{(m)}(1)%
		}{m!}(e^{w}-1)^{m} \\
		&= f_\nu(1)\,w^{\nu-1}\frac{we^{xw}}{e^{w}-1}+w^{\nu}e^{xw}\sum_{m\ge1}\frac{%
			f_\nu^{(m)}(1)}{m!}(e^{w}-1)^{m-1}.
	\end{align*}
	The first summand contributes $f_\nu(1)(n)_{\nu-1}B_{n-\nu+1}(x)$, exactly
	as in the proof of Theorem \ref{thm:pf}. For the second, by \eqref{stirling} 
	\begin{equation*}
		\frac{1}{(m-1)!}(e^{w}-1)^{m-1}=\sum_i S(i,m-1)\frac{w^{i}}{i!}
	\end{equation*}
	and 
	\begin{equation*}
		\frac{1}{m!}(e^w-1)^{m-1}=\frac1m\cdot\frac{(e^w-1)^{m-1}}{(m-1)!},
	\end{equation*}
	so that 
	\begin{equation*}
		w^{\nu}e^{xw}\frac{(e^{w}-1)^{m-1}}{m!}=\frac{w^{\nu}}{m}\sum_{k\ge0}\Bigg(%
		\sum_{i=0}^{k}\binom ki S(i,m-1)x^{k-i}\Bigg)\frac{w^{k}}{k!},
	\end{equation*}
	and the coefficient of $w^{n}/n!$ is 
	\begin{equation*}
		\frac{(n)_{\nu}}{m}\sum_{i}\binom{n-\nu}{i}S(i,m-1)x^{n-\nu-i}.
	\end{equation*}
	As $S(i,m-1)=0$ for $i<m-1$, the terms with $m-1>n-\nu$ vanish, whence the
	range of~ $m$.
\end{proof}

We now make explicit the cases $n=\nu$ and $n=\nu+1$, whose form is not
apparent from \eqref{eq:main-pf}. Write $H=H_{\nu-2}$, $H^{(2)}=H^{(2)}_{%
	\nu-2}$.

\begin{corollary}
	\label{cor:small} For $\nu\ge2$, 
	\begin{align}
		\ell_{\nu-1}(x;\nu)&=(-1)^{\nu}(\nu-1),  \label{eq:lvm1} \\
		\ell_{\nu}(x;\nu)&=(-1)^{\nu}\nu(\nu-1)\left(x+H-\frac32\right),
		\label{eq:lv} \\
		\ell_{\nu+1}(x;\nu)&=(-1)^{\nu}\frac{(\nu+1)\nu(\nu-1)}{2}%
		\left(x^{2}+(2H-3)x+H^{2}-2H+H^{(2)}+\frac{13}{6}\right).  \label{eq:lvp1}
	\end{align}
	More generally, for $k\ge0$ the polynomial $\ell_{\nu+k}(x;\nu)$ has degree $%
	k+1$ in $x$, leading coefficient $(-1)^\nu\binom{\nu+k}{\nu-1}(\nu-1)$, and
	its coefficients lie in $\mathbb{Q}[H^{(1)}_{\nu-2},\dots,H^{(k+1)}_{\nu-2}]$%
	.
\end{corollary}

\begin{proof}
	Formula \eqref{eq:lvm1} is part of Theorem \ref{thm:pf}. For $n=\nu$ formula %
	\eqref{eq:taylor} has the single term $m=1$, $i=0$, so 
	\[\ell_\nu(x;\nu)=\nu! 
	\,f_\nu(1)B_1(x)+\nu!\,f_\nu^{\prime}(1)=\nu!A_1\Bigl(x-\frac12+H-1\Bigr),\]
	which is \eqref{eq:lv} since $\nu!/(\nu-2)!=\nu(\nu-1)$. For $n=\nu+1$: $%
	(\nu+1)_{\nu-1}=(\nu+1)!/2$, $B_2(x)=x^2-x+\tfrac16$, and the double sum is $%
	f_\nu^{\prime }(1)\,x+\tfrac12 f_\nu^{\prime \prime }(1)$ (using $%
	S(0,0)=S(1,1)=1$, $S(1,0)=0$); so 
	\[
	\ell_{\nu+1}(x;\nu)=(\nu+1)!\,A_1\left[\frac12\Big(x^2-x+\frac16\Big) 
	+(H-1)x+\frac12\Big((H-1)^2+1+H^{(2)}\Big)\right],
	\]
	and $(\nu+1)!A_1/2=(-1)^\nu(\nu+1)\nu(\nu-1)/2$. For the last statement
	observe that $G_\nu(w;x)=e^{xw}G_\nu(w;0)$ gives $\ell_n(x;\nu)=\sum_{i}%
	\binom ni x^{n-i}\ell_i(\nu)$; the lowest non--vanishing $\ell_i(\nu)$ is $%
	i=\nu-1$, which gives the degree and the leading coefficient. The statement
	about the coefficients follows from \eqref{eq:taylor}, since $%
	\ell_{\nu+k}(x;\nu)$ only involves $f_\nu^{(m)}(1)$ with $m\le k+1$, and $%
	Y_m(c_1,\dots,c_m)\in\mathbb{Q}[H^{(1)}_{\nu-2},\dots,H^{(m)}_{\nu-2}]$ by %
	\eqref{eq:fder}.
\end{proof}

As a check, for $\nu=4$ ($H=3/2$, $H^{(2)}=5/4$) we get $\ell_4(4)=12%
\cdot0=0 $ and 
\begin{equation*}
	\ell_5(4)=30\cdot\Big(\frac94-3+\frac54+\frac{13}6\Big)=30\cdot\frac83=80,
\end{equation*}
in agreement with Table \ref{tab:values}; for $\nu=7$, 
\begin{equation*}
	\ell_7(7)=-42\Big(\frac{137}{60}-\frac32\Big)=-\frac{329}{10}.
\end{equation*}
The vanishing $\ell_4(4)=0$ is thus explained by $H_2=3/2$; since $H_m=3/2$
only for $m=2$ ($H_1=1$, $H_3={11}/{6}$, and for $m\ge4$ the denominator of $%
H_m$ is divisible by $4$, because the term $2^{-\lfloor\log_2 m\rfloor}$ is
the only one whose denominator carries the maximal power of $2$ and cannot
be cancelled), $\nu=4$ is the only $\nu\ge3$ for which $\ell_\nu(\nu)=0$.
This vanishing is not apparent from \eqref{eq:main-pf}.

\begin{remark}
	Comparing \eqref{eq:main-pf} and \eqref{eq:taylor} at $x=0$, $n=\nu+k$, and
	using $F_k(-1)=(-1)^k$ one obtains the identity 
	\begin{equation*}
		\sum_{\substack{ j=0  \\ j\ne1}}^{\nu-1}\frac{A_j}{(j-1)^{m+1}}=-\frac{%
			f_\nu^{(m+1)}(1)}{(m+1)!}\qquad(m\ge0),
	\end{equation*}
	which is the Taylor expansion of $(u-1)\sum_{j\ne1}A_j/(u-j)$ at $u=1$. It
	expresses, e.g.\ for $m=0$, the harmonic number $H_{\nu-2}$ as a signed
	binomial sum, 
	\begin{equation*}
		\sum_{j\ne1}\frac{(-1)^{j}}{j-1}\binom{\nu-1}{j}=(\nu-1)(H_{\nu-2}-1),
	\end{equation*}
	a formula of the type that appears in \cite{SimsekSurvey} in connection with
	the numbers $y_1(n,k;\lambda)$; we do not pursue this here.
\end{remark}

Since $G_\nu(w;x)=e^{xw}G_\nu(w;0)$, the polynomials $\ell_n(x;\nu)$ form an
Appell sequence, $\frac{\partial}{\partial x}\ell_n(x;\nu)=n\,\ell_{n-1}(x;%
\nu)$, and 
\begin{equation*}
	\ell_n(x+y;\nu)=\sum_k\binom nk\ell_k(x;\nu)y^{n-k}.
\end{equation*}
The behaviour under the forward difference $\Delta_x\varphi(x)=\varphi(x+1)-%
\varphi(x)$, which for the Bernoulli polynomials is $\Delta_xB_n(x)=nx^{n-1}$%
, is less immediate; here the function $f_\nu$ of \eqref{def-f} enters again.

\begin{proposition}
	\label{prop:diff} For $n\ge\nu\ge2$ and every $x$, 
	\begin{equation}  \label{eq:diff}
		\ell_n(x+1;\nu)-\ell_n(x;\nu)=(n)_{\nu}\sum_{m=0}^{n-\nu}f_\nu^{(m)}(1)%
		\sum_{i=m}^{n-\nu}\binom{n-\nu}{i}S(i,m)\,x^{\,n-\nu-i},
	\end{equation}
	and $\ell_n(x+1;\nu)=\ell_n(x;\nu)$ for $n\le \nu-1$. Consequently, for $N\in%
	\mathbb{N}$, 
	\begin{equation}  \label{eq:faulhaber}
		\sum_{m=0}^{n-\nu}f_\nu^{(m)}(1)\sum_{i=m}^{n-\nu}\binom{n-\nu}{i}%
		S(i,m)\sum_{r=0}^{N-1}r^{\,n-\nu-i}=\frac{\ell_n(N;\nu)-\ell_n(\nu)}{%
			(n)_{\nu}}.
	\end{equation}
\end{proposition}

\begin{proof}
	We have 
	\begin{equation*}
		G_\nu(w;x+1)-G_\nu(w;x)=w^{\nu}e^{xw}\frac{e^{w}-1}{(e^{w})_{\nu}}%
		=w^{\nu}e^{xw}f_\nu(e^{w})=w^{\nu}e^{xw}\sum_{m\ge0}\frac{f_\nu^{(m)}(1)}{m!}%
		(e^{w}-1)^{m}
	\end{equation*}
	and 
	\begin{equation*}
		\frac{(e^{w}-1)^{m}}{m!}e^{xw}=\sum_k\biggl(\sum_i\binom ki S(i,m)x^{k-i}%
		\biggr)\frac{w^{k}}{k!}
	\end{equation*}
	by \eqref{stirling}; the coefficient of $w^{n}/n!$ is then read off as in
	the proof of Theorem \ref{thm:taylor}. Summing \eqref{eq:diff} over $%
	x=0,1,\dots,N-1$ gives \eqref{eq:faulhaber}.
\end{proof}

Formula \eqref{eq:faulhaber} is a Faulhaber--type identity: the left side is
a fixed linear combination (with harmonic--number coefficients) of the power
sums $\sum_{r<N}r^{k}$, $0\le k\le n-\nu$, and the right side is a
polynomial in $N$ of degree $n-\nu+1$ given by \eqref{eq:taylor}. For $n=\nu$
it degenerates to $f_\nu(1)\,N=(\ell_\nu(N;\nu)-\ell_\nu(\nu))/\nu!$, which
is \eqref{eq:lv}; for $n=\nu+1$ it reads 
\begin{equation*}
	\frac{1}{2}f_\nu(1)\,N(N-1)+f_\nu^{\prime }(1)\,N=\frac{1}{(\nu+1)!}\bigl(%
	\ell_{\nu+1}(N;\nu)-\ell_{\nu+1}(\nu)\bigr) ,
\end{equation*}
in agreement with \eqref{eq:lvp1}. Compare \cite{SSS2024}, where power sums
are treated by means of moment generating functions of discrete
distributions; Proposition \ref{prop:moments} below explains why the same
objects occur here.

\section{Recurrence relations}

\label{sec:rec}
As is the case with special polynomials, recurrence relations for the polynomials  $\ell_n(x;\nu)$ are of great importance. These relations enable the rapid and accurate computation of high-degree polynomials without the need to directly calculate cumbersome algebraic expansions. The motivation of this section is to give with the aid of generating function.

\begin{proposition}
	\label{prop:rec1} For $\nu\in\mathbb{N}$, $n\in\mathbb{N}_0$ and every $x$, 
	\begin{equation}  \label{eq:rec1}
		\sum_{k=0}^{n}\binom
		nk\ell_k(x;\nu+1)-\nu\,\ell_n(x;\nu+1)=n\,\ell_{n-1}(x;\nu),
	\end{equation}
	where $\ell_{-1}:=0$.
\end{proposition}

\begin{proof}
	Clearly $(e^{w}-\nu)G_{\nu+1}(w;x)=wG_\nu(w;x)$. The coefficient of $w^n/n!$
	on the left is 
	\begin{equation*}
		\sum_k\binom nk\ell_k(x;\nu+1)-\nu\,\ell_n(x;\nu+1),
	\end{equation*}
	on the right it is $n\,\ell_{n-1}(x;\nu)$.
\end{proof}

\begin{proposition}
	\label{prop:rec2} For $\nu\in\mathbb{N}$, $n\in\mathbb{N}_0$, 
	\begin{equation}  \label{eq:rec2}
		\sum_{k=1}^{\nu}s(\nu,k)\sum_{m=0}^{n}\binom nm
		k^{\,n-m}\ell_m(x;\nu)=(n)_{\nu}\,x^{\,n-\nu},
	\end{equation}
	the right hand side being $0$ for $n<\nu$ and $\nu!$ for $n=\nu$, $x=0$.
\end{proposition}

\begin{proof}
	By \eqref{stirling}, $(e^{w})_{\nu}=\sum_{k=0}^{\nu}s(\nu,k)e^{kw}$ (with $%
	s(\nu,0)=0$ for $\nu\ge1$), so that 
	\begin{equation*}
		\sum_{k=0}^{\nu} s(\nu,k)\,e^{kw}\,G_\nu(w;x)=w^{\nu}e^{xw};
	\end{equation*}
	expand both sides and compare coefficients.
\end{proof}

Either recurrence determines the numbers completely: \eqref{eq:rec1}
together with $\ell_n(x;1)=n(x-1)^{n-1}$ (a two--index scheme, which was
used for the numerical verification in this paper), and \eqref{eq:rec2} for
each fixed $\nu$ together with the initial value \eqref{eq:lvm1}. Note that
the coefficient of $\ell_n$ in \eqref{eq:rec2} is $\sum_k
s(\nu,k)=(1)_{\nu}=0$ for $\nu\ge2$, so that \eqref{eq:rec2} is really a
recurrence for $\ell_{n-1}$ with coefficient 
\begin{equation*}
	n\sum_k s(\nu,k)k=n\,\frac{d}{du}(u)_{\nu}\bigg|_{u=1}=n(-1)^{\nu}(\nu-2)!%
	\ne0.
\end{equation*}
For $\nu=2$ (where $s(2,1)=-1$, $s(2,2)=1$) relation \eqref{eq:rec2} reads 
\begin{equation*}
	\sum_{m=0}^{n}\binom nm\big(2^{n-m}-1\big)\ell_m(x;2)=(n)_{2}x^{n-2}.
\end{equation*}

\section{An interpolation function and geometric moments}

\label{sec:interp}

The Bernoulli polynomials are interpolated by the Hurwitz zeta function, $%
B_{k+1}(x)=-(k+1)\zeta (-k,x)$, and the Apostol--Bernoulli polynomials by
Lerch's transcendent \cite{Apostol1951}. The nodes $j\geq 2$ in %
\eqref{eq:G-split} correspond, as we now show, to Lerch's function at the
points $z=1/j$, and the node $j=0$ to the elementary function $(x-1)^{-s}$.
For the statement we take $x>1$ real, which keeps all powers unambiguous;
other choices of $x$ are discussed after the proof.

An interpolation function for the polynomials $\ell _{n}(x;\nu )$ was
introduced recently by Pokaz, Ribi\v{c}i\'{c} Penava and Simsek \cite%
{Peneva1}, namely%
\begin{equation}
	\begin{split}
		Z(s;x,\nu )& =-\frac{s(s+1)\cdots (s+\nu -1)}{(\nu -1)!}\sum_{m=0}^{\infty }%
		\frac{1}{\left( x-1+m\right) ^{s+\nu }} \\
		& \qquad \times \sum_{0\leq k_{1}\leq k_{2}\leq \cdots \leq k_{\nu -2}\leq m}%
		\frac{1}{1^{k_{1}}2^{k_{2}-k_{1}}3^{k_{3}-k_{2}}\cdots (\nu -1)^{m-k_{\nu -2}}}
	\end{split}
	\label{DM-1}
\end{equation}%
for $\nu \geq 3$, and $Z(s;x,2)=-s(s+1)\zeta (s+2,x-1)$, where $x>1$. The
series converges absolutely for $\operatorname{Re}(s+\nu )>1$; its analytic
continuation and its values at the negative integers are studied in \cite%
{Peneva1}. For $\nu =3$, for instance, the inner sum is $2-2^{-m}$ and%
\begin{equation*}
	Z(s;x,3)=-s(s+1)(s+2)\left( \zeta (s+3,x-1)-\frac{1}{2}\Phi \left( \frac{1}{2%
	},s+3,x-1\right) \right) .
\end{equation*}%
The function $Z_{\nu }(s;x)$ constructed in the following theorem is a
finite combination of the Hurwitz zeta function and of $\nu -2$ Lerch
functions; Proposition \ref{prop:PRS} shows that $Z(s;x,\nu )$ is an
elementary multiple of $Z_{\nu }(s+\nu ;x)$, so that the $(\nu -2)$--fold
sum in \eqref{DM-1} collapses to these $\nu -1$ functions.

\begin{theorem}
	\label{thm:interp} Let $\nu\ge2$ and $x>1$. Define, for $s\in\mathbb{C}$, 
	\begin{equation}  \label{def-Z}
		Z_\nu(s;x):=\frac{(-1)^{\nu-1}}{(\nu-1)!}(x-1)^{-s}+\frac{(-1)^{\nu-1}}{%
			(\nu-2)!}\zeta(s,x)-\sum_{j=2}^{\nu-1}\frac{(-1)^{\nu-1-j}}{j\cdot
			j!\,(\nu-1-j)!}\,\Phi\!\left(\frac1j,s,x\right).
	\end{equation}
	Then $Z_\nu(\,\cdot\,;x)$ is meromorphic on $\mathbb{C}$ with a single
	simple pole at $s=1$, of residue $(-1)^{\nu-1}/(\nu-2)!$, and 
	\begin{equation}  \label{eq:interp}
		\ell_n(x;\nu)=(n)_{\nu}\,Z_\nu(\nu-n;x)\qquad(n\ge \nu).
	\end{equation}
\end{theorem}

\begin{proof}
	For $j\ge2$, $0<1/j<1$, and $\Phi(1/j,s,x)=\sum_r j^{-r}(r+x)^{-s}$
	converges absolutely and locally uniformly in $s\in\mathbb{C}$ (for $x>0$),
	so it is entire; $\zeta(s,x)$ is meromorphic with the only pole $s=1$,
	residue $1$; $(x-1)^{-s}$ is entire. This gives the analytic statement (note
	that the coefficient of $\zeta$ is $-A_1$).
	
	For \eqref{eq:interp}, fix $j\ge2$ and let $\operatorname{Re}w<\log j$. Then $%
	|e^{w}/j|<1$ and 
	\begin{equation*}
		\frac{e^{xw}}{e^{w}-j}=-\frac1j\cdot\frac{e^{xw}}{1-e^{w}/j}%
		=-\sum_{r=0}^{\infty}j^{-r-1}e^{(r+x)w}.
	\end{equation*}
	Every term is entire in $w$, and the series converges uniformly on compact
	subsets of $\{\operatorname{Re}w<\log j\}$, so we may differentiate termwise at $w=0$%
	: the coefficient of $w^{k}/k!$ is 
	\begin{equation*}
		-j^{-1}\sum_r j^{-r}(r+x)^{k}=-j^{-1}\Phi(1/j,-k,x).
	\end{equation*}%
	(Comparing with Lemma \ref{lem:fubini} one obtains $\Phi(1/j,-k,x)=(j/(
	j-1))F_k(x;{1}/(j-1))$, a formula of Euler type; compare \cite[Section 3]%
	{Boyadzhiev}.) Hence the coefficient of $w^n/n!$ in $w^{%
		\nu}A_je^{xw}/(e^{w}-j)$ is 
	\begin{equation*}
		-(n)_{\nu}\frac{A_j}{j}\Phi(1/j,\nu-n,x)
	\end{equation*}
	for $n\ge \nu$. The $j=0$ term is 
	\begin{equation*}
		A_0(n)_{\nu}(x-1)^{n-\nu}=A_0(n)_{\nu}(x-1)^{-(\nu-n)},
	\end{equation*}
	and the $j=1$ term is 
	\begin{equation*}
		A_1(n)_{\nu-1}B_{n-\nu+1}(x)=-A_1(n)_{\nu-1}(n-\nu+1)\zeta(%
		\nu-n,x)=-A_1(n)_{\nu}\zeta(\nu-n,x),
	\end{equation*}
	valid for $n-\nu+1\ge1$. Adding the three kinds of contributions and
	inserting the values of $A_0,A_1,A_j$ from Lemma \ref{lem:pf} gives %
	\eqref{eq:interp}.
\end{proof}

For $0<x<1$ the only change is in the $j=0$ term, which should be read as $%
A_0\cos(\pi s)\,(1-x)^{-s}$ (this agrees with $A_0(x-1)^{n-\nu}$ at $%
s=\nu-n\in\mathbb{Z}$); for $x=0$ one drops the term $r=0$ of the Lerch
series, i.e., replaces $\Phi(1/j,s,0)$ by $Li_s(1/j)$, and the $j=0$ term
becomes $A_0\cos(\pi s)$: 
\begin{equation}  \label{def-Z0}
	Z_\nu(s):=\frac{(-1)^{\nu-1}}{(\nu-1)!}\cos(\pi s)+\frac{(-1)^{\nu-1}}{%
		(\nu-2)!}\zeta(s)-\sum_{j=2}^{\nu-1}\frac{(-1)^{\nu-1-j}}{j\cdot
		j!\,(\nu-1-j)!}Li_s\!\left(\tfrac1j\right),
\end{equation}
and then $\ell_n(\nu)=(n)_{\nu}Z_\nu(\nu-n)$ for $n>\nu$, where $n=\nu$ is
excluded because the omitted term $r=0$ contributes $\sum_{r%
	\ge0}r^{0}j^{-r}=1+Li_0(1/j)$ and $B_1(0)\ne B_1(1)$; the value $%
\ell_\nu(\nu)$ is given by \eqref{eq:lv}. We checked \eqref{eq:interp} and %
\eqref{def-Z0} numerically for $2\le \nu\le6$ and several $n$ and $x$ with a
multiprecision implementation of $\zeta(s,x)$ and $\Phi$.

The values of $Z_\nu$ at $s=0$ and $s=-1$ are \eqref{eq:lv} and %
\eqref{eq:lvp1}. The same function is also defined at positive integers $%
s=k\ge2$, where $\zeta(k,x)$ and $\Phi(1/j,k,x)$ are functions of
polylogarithmic type. For example, at $x=2$, $\nu=3$ (where $A_0=1/2$, $%
A_1=-1$, $A_2=1/2$, $\zeta(2,2)=\zeta(2)-1$ and $\Phi(1/2,2,2)=4Li_2(1/2)-2$%
), 
\begin{equation*}
	Z_3(2;2)=\frac12+\zeta(2)-1-\frac14\biggl(4Li_2\Bigl(\frac12\Bigr)-2\biggr)=%
	\frac{\pi^2}{6}-Li_2\Bigl(\frac12\Bigr)=\frac{\pi^2}{12}+\frac12\log^2 2 ,
\end{equation*}
This indicates that the values of $Z_\nu$ at positive integers involve
constants of a different nature ($\log 2$, $\zeta(3)$, $Li_k(1/j)$) from
the rational values at non--positive integers. We have not investigated
whether these values have an interpretation in terms of the numbers $%
\ell_n(\nu)$.

The function $Z_{\nu }$ has, like the Hurwitz zeta function, an Euler--type
integral representation, in which the generating function \eqref{def-ell}
itself, evaluated at $-t$, appears under the integral sign. In this sense $%
Z_{\nu }$ is the Mellin transform of the function $\,G_{\nu }(-t;x)$.

\begin{proposition}
	\label{prop:mellin} For $\nu\ge2$, $x>1$ and $\operatorname{Re}s>1$, 
	\begin{equation}  \label{eq:mellin}
		Z_\nu(s;x)=\frac{1}{\Gamma(s)}\int_0^{\infty}\frac{t^{s-1}e^{-xt}}{%
			(e^{-t})_{\nu}}\,dt=\frac{(-1)^{\nu}}{\Gamma(s)}\int_0^{\infty}t^{s-1-\nu}%
		\,G_\nu(-t;x)\,dt .
	\end{equation}
\end{proposition}

\begin{proof}
	By Lemma \ref{lem:pf} with $u=e^{-t}$, 
	\begin{equation*}
		\frac{e^{-xt}}{(e^{-t})_{\nu}}=A_0e^{-(x-1)t}-\frac{A_1e^{-xt}}{1-e^{-t}}%
		-\sum_{j=2}^{\nu-1}\frac{A_j}{j}\cdot\frac{e^{-xt}}{1-e^{-t}/j}.
	\end{equation*}
	For $\operatorname{Re}s>1$ and $x>0$ one has the classical integrals $%
	\Gamma(s)\zeta(s,x)=\int_0^\infty t^{s-1}e^{-xt}(1-e^{-t})^{-1}dt$ and $%
	\Gamma(s)\Phi(z,s,x)=\int_0^\infty t^{s-1}e^{-xt}(1-ze^{-t})^{-1}dt$ ($|z|<1$%
	), obtained by expanding the geometric series and integrating termwise, and $%
	\int_0^\infty t^{s-1}e^{-(x-1)t}dt=\Gamma(s)(x-1)^{-s}$ for $x>1$. All three
	integrals converge absolutely (near $t=0$ the integrand of the first behaves
	like $t^{s-2}$), so the sum of the three is \eqref{def-Z}. The second form
	of \eqref{eq:mellin} is just $G_\nu(-t;x)=(-t)^{\nu}e^{-xt}/(e^{-t})_{\nu}$.
\end{proof}

From \eqref{eq:mellin} the standard Hankel--contour argument (as in the
proof of $\zeta(-n,x)=-\frac{B_{n+1}(x)}{n+1}$ in \cite[Chapter 12]%
{Apostol1976}) gives a second proof of Theorem \ref{thm:interp}: if $%
h(t)=\sum_{m\ge-1}h_mt^{m}$ near $t=0$ and $Z(s)=\Gamma(s)^{-1}\int_0^\infty
t^{s-1}h(t)dt$, then $Z$ continues meromorphically with the only possible
pole at $s=1$ with residue $h_{-1}$, and $Z(-m)=(-1)^{m}m!\,h_m$ for $m\in%
\mathbb{N}_0$. Here 
\begin{equation*}
	h(t)=\frac{e^{-xt}}{(e^{-t})_{\nu}}=(-1)^{\nu}\sum_{n\ge
		\nu-1}(-1)^{n}\ell_n(x;\nu)\,\frac{t^{\,n-\nu}}{n!},
\end{equation*}
so $h_{-1}=-\ell_{\nu-1}(x;\nu)/(\nu-1)!=(-1)^{\nu-1}/(\nu-2)!$ and $%
h_m=(-1)^{m}\ell_{m+\nu}(x;\nu)/(m+\nu)!$, i.e., $Z_\nu(-m;x)=m!\,\ell_{m+%
	\nu}(x;\nu)/(m+\nu)!$, which is \eqref{eq:interp}. The partial fraction
proof is more direct, but the integral representation \eqref{eq:mellin} does
not refer to the individual nodes and is the appropriate starting point for
treating nodes with multiplicity.

The integral \eqref{eq:mellin} also identifies $Z_{\nu }$ with the function %
\eqref{DM-1}. Let $h_{m}(y_{1},\dots ,y_{r})$ denote the complete homogeneous
symmetric polynomial of degree $m$, i.e.%
\begin{equation*}
	\prod_{i=1}^{r}\frac{1}{1-y_{i}z}=\sum_{m=0}^{\infty }h_{m}(y_{1},\dots
	,y_{r})z^{m}.
\end{equation*}

\begin{proposition}
	\label{prop:PRS} Let $\nu \geq 2$ and $x>1$. For $\operatorname{Re}s>1$,%
	\begin{equation}
		Z_{\nu }(s;x)=\frac{(-1)^{\nu -1}}{(\nu -1)!}\sum_{m=0}^{\infty }h_{m}\left(
		1,\frac{1}{2},\dots ,\frac{1}{\nu -1}\right) (m+x-1)^{-s},
		\label{eq:Zdirichlet}
	\end{equation}%
	and the inner sum in \eqref{DM-1} equals $h_{m}(1,\frac{1}{2},\dots ,\frac{1%
	}{\nu -1})$. Consequently%
	\begin{equation}
		Z(s;x,\nu )=(-1)^{\nu }\,s(s+1)\cdots (s+\nu -1)\,Z_{\nu }(s+\nu ;x),
		\label{eq:ZZ}
	\end{equation}%
	$Z(s;x,\nu )$ is an entire function of $s$, $Z(-n;x,\nu )=\ell _{n}(x;\nu )$
	for every $n\in \mathbb{N}_{0}$, and%
	\begin{equation}
		\begin{split}
			\sum_{m=0}^{\infty }h_{m}\left( 1,\frac{1}{2},\dots ,\frac{1}{\nu -1}\right)
			(m+x-1)^{-s}& =(x-1)^{-s}+(\nu -1)\zeta (s,x) \\
			& \qquad -\sum_{j=2}^{\nu -1}(-1)^{j}\binom{\nu -1}{j}\frac{1}{j}\,\Phi
			\left( \frac{1}{j},s,x\right) .
		\end{split}
		\label{eq:closed}
	\end{equation}
\end{proposition}

\begin{proof}
	For $t>0$,%
	\begin{equation*}
		(e^{-t})_{\nu }=e^{-t}\prod_{j=1}^{\nu -1}(e^{-t}-j)=(-1)^{\nu -1}(\nu
		-1)!\,e^{-t}\prod_{j=1}^{\nu -1}\left( 1-\frac{e^{-t}}{j}\right) ,
	\end{equation*}%
	and $\prod_{j=1}^{\nu -1}(1-e^{-t}/j)^{-1}=\sum_{m\geq 0}h_{m}(1,\frac{1}{2}%
	,\dots ,\frac{1}{\nu -1})e^{-mt}$ by the definition of $h_{m}$. Writing $%
	h_{m}$ as the sum of the monomials $1^{-e_{1}}2^{-e_{2}}\cdots (\nu
	-1)^{-e_{\nu -1}}$ over $e_{1}+\dots +e_{\nu -1}=m$ and putting $%
	k_{i}=e_{1}+\dots +e_{i}$ gives the inner sum of \eqref{DM-1}. Since $0\leq
	h_{m}\leq \prod_{j=2}^{\nu -1}(1-\frac{1}{j})^{-1}=\nu -1$, the series $%
	\sum_{m}h_{m}t^{s-1}e^{-(m+x-1)t}$ may be integrated termwise in %
	\eqref{eq:mellin} for $\operatorname{Re}s>1$, and \eqref{eq:Zdirichlet}
	follows. Comparing \eqref{eq:Zdirichlet} with \eqref{DM-1} gives %
	\eqref{eq:ZZ} for $\operatorname{Re}(s+\nu )>1$, hence for all $s$ by
	analytic continuation; the factor $s+\nu -1$ cancels the pole of $Z_{\nu
	}(s+\nu ;x)$ at $s=1-\nu $, so $Z(s;x,\nu )$ is entire. For $n\geq \nu $, %
	\eqref{eq:ZZ} at $s=-n$ is $(n)_{\nu }Z_{\nu }(\nu -n;x)=\ell _{n}(x;\nu )$
	by \eqref{eq:interp}; for $0\leq n\leq \nu -2$ the product $s(s+1)\cdots
	(s+\nu -1)$ vanishes at $s=-n$, and so does $\ell _{n}(x;\nu )$; for $n=\nu
	-1$ the vanishing factor $s+\nu -1$ meets the pole of $Z_{\nu }$, of residue $%
	(-1)^{\nu -1}/(\nu -2)!$, and the value is $(-1)^{\nu }\cdot (-1)^{\nu
		-1}(\nu -1)!\cdot (-1)^{\nu -1}/(\nu -2)!=(-1)^{\nu }(\nu -1)=\ell _{\nu
		-1}(x;\nu )$. Finally \eqref{eq:closed} is $(-1)^{\nu -1}(\nu -1)!$ times %
	\eqref{def-Z}.
\end{proof}

Thus the entire function $Z(s;x,\nu )$ interpolates all the polynomials $\ell
_{n}(x;\nu )$, $n\geq 0$, while $Z_{\nu }$ does so for $n\geq \nu $ up to the
factor $(n)_{\nu }$, in the same way as $\zeta (s,x)$ interpolates $%
B_{k+1}(x)$ up to the factor $-(k+1)$. The case $\nu =2$ of \eqref{eq:ZZ} is $%
Z(s;x,2)=s(s+1)Z_{2}(s+2;x)$ with $Z_{2}(s;x)=-\zeta (s,x-1)$.

We now turn to the probabilistic interpretation. Recall \cite%
{SimsekSurvey,SimsekKilar2025} that $X$ is geometric with parameter $%
p\in(0,1)$ if $P(X=r)=pq^{r-1}$, $r\in\mathbb{N}$, $q=1-p$, and that its
moments are 
\begin{equation}  \label{eq:moments}
	\mu_m(p,q)=E( X^{m})=\sum_{r\ge1}r^{m}pq^{r-1}=-\frac{p\,\mathcal{B}_{m+1}(1;q)}{m+1}
\end{equation}
(the last equality is from \cite{BSimsek2024}, see also \cite{SimsekSurvey};
it is equivalent to \eqref{eq:AB-F}). Let $X_j$, $j\ge2$, be geometric with $%
q_j=1/j$, $p_j=1-1/j$, and let $X_0:=0$ (the degenerate case corresponding
to $q=0$). Then for every $k\in\mathbb{N}_0$ and every $x$ 
\begin{equation*}
	E\big[(X_j+x-1)^{k}\big]=\sum_{r\ge1}(r+x-1)^{k}p_jq_j^{r-1}=p_j%
	\Phi(q_j,-k,x)=\frac{j-1}{j}\,\Phi\!\left(\frac1j,-k,x\right),
\end{equation*}
and by the proof of Theorem \ref{thm:interp} the node $j$ contributes 
\begin{equation*}
	-(n)_{\nu}\frac{A_j}{j}\Phi\!\left(\frac1j,\nu-n,x\right)=-(n)_{\nu}\frac{A_j%
	}{j-1}E\big[(X_j+x-1)^{n-\nu}\big];
\end{equation*}
for $j=0$ the contribution 
\begin{equation*}
	A_0(n)_{\nu}(x-1)^{n-\nu}=-(n)_{\nu}\frac{A_0}{0-1}E[(X_0+x-1)^{n-\nu}]
\end{equation*}
has the same form. This proves the following statement.

\begin{proposition}
	\label{prop:moments} For $ n\ge\nu\ge2$ and $x\in\mathbb{C}$, 
	\begin{align}  \label{eq:prop-moments}
		\ell_n(x;\nu)=&\frac{(-1)^{\nu}}{(\nu-2)!}(n)_{\nu-1}B_{n-\nu+1}(x) \\
		&\qquad -(n)_{\nu}\sum_{\substack{ j=0  \\ j\ne1}}^{\nu-1}\frac{%
			(-1)^{\nu-1-j}}{j!\,(\nu-1-j)!\,(j-1)}\,E\!\left[(X_j+x-1)^{n-\nu}\right]. 
		\notag
	\end{align}
	In particular, at $x=1$ the numbers $\ell_n(1;\nu)$ are, up to the Bernoulli
	term, a fixed linear combination of the $(n-\nu)$-th moments of geometric
	random variables with success probabilities $\tfrac12,\tfrac23,\dots,\tfrac{%
		\nu-2}{\nu-1}$.
\end{proposition}

In \cite{Peneva}, Ribi\v{c}i\'{c} Penava, Pokaz and Simsek derived various
identities, finite sums and moment formulas involving the polynomials $\ell
_{n}(x;\nu )$, the Frobenius--Euler numbers, the Stirling numbers, the
Bernoulli polynomials and the Eulerian polynomials.

So the numbers $\ell _{n}(\nu )$ encode, in a single generating function,
the moments of a whole family of geometric distributions, in the same way as
the Fubini numbers encode the moments of the geometric distribution with $p=%
\frac{1}{2}$ (cf.\ \cite{SimsekKilar2025}). One may turn %
\eqref{eq:prop-moments} around and use it, together with \eqref{eq:taylor},
to express such moment combinations through Stirling and harmonic numbers;
we do not pursue this here.

\section{Asymptotics}
\label{sec:asymp}

It is well known that asymptotic analysis of special polynomials has many important applications. Asymptotic methods offer a powerful way to approximate the behavior of these polynomials when their degrees or variables reach very large values. Calculating the exact values of high-degree polynomials have significant mathematical challenges in applications of real-world problems. Therefore, asymptotic formulas allow complex expressions to be replaced by simple, highly accurate approximations. The aim of this section is to present asymptotic analyses and formulas for $\ell_n(x;\nu)$ polynomials.

Table \ref{tab:values} suggests that $|\ell_n(\nu)|$ grows roughly like $n!$
times a geometric factor, with a sign depending on $\nu$ only. This is made
precise by singularity analysis \cite[Ch.~IV]{FlajoletSedgewick}.

\begin{proposition}
	\label{prop:asymp} Let $\nu\ge3$ and $x\in\mathbb{C}$. Then, as $n\to\infty$%
	, 
	\begin{equation}  \label{eq:asymp}
		\ell_n(x;\nu)=(-1)^{\nu}\,\frac{2^{x}\,(\log2)^{\nu-n-1}}{4\,(\nu-3)!}\;n!\,%
		\big(1+O(\rho^{n})\big),\qquad \rho=%
		\begin{cases}
			\log2/\log3, & \nu\ge4, \\ 
			\log2/(2\pi), & \nu=3.%
		\end{cases}%
	\end{equation}
	For $\nu=2$, $\ell_n(2)=(-1)^{n-1}(n)_{2}+nB_{n-1}$ and the Bernoulli term
	dominates: $\ell_{2m+1}(2)\sim (-1)^{m-1}\frac{2\,(2m+1)!}{(2\pi)^{2m}}$, $%
	\ell_{2m}(2)=-(2m)_{2}$ exactly for $m\ge2$.
\end{proposition}

\begin{proof}
	$G_\nu(w;x)$ is meromorphic in $\mathbb{C}$, with poles at the solutions of $%
	e^{w}=j$, $j=1,\dots,\nu-1$, i.e., at $w=\log j+2\pi i k$, $k\in\mathbb{Z}$,
	except that $w=0$ is removable (the numerator $w^{\nu}$ kills the simple
	zero of $e^{w}-1$). For $\nu\ge3$ the node $j=2$ is present and the pole
	nearest to the origin is $w_0=\log2$, simple, with 
	\begin{equation*}
		Res_{w=\log 2}G_\nu(w;x)=\frac{(\log2)^{\nu}2^{x}}{\frac{d}{dw}(e^{w}-2)%
			\big|_{\log2}\prod_{j\ne2}(2-j)}=\frac{(\log2)^{\nu}2^{x}}{%
			2\cdot2\cdot1\cdot(-1)^{\nu-3}(\nu-3)!}=:R .
	\end{equation*}
	The next poles in modulus are $\log3$ (if $\nu\ge4$) and $\pm2\pi i$, $\log
	2\pm2\pi i$ (if $\nu=3$). Hence $G_\nu(w;x)-R/(w-\log2)$ is analytic in $%
	|w|<r$ for every $r<\log3$ (resp.\ $r<2\pi$), so its $n$-th Taylor
	coefficient is $O(r^{-n})$, while the $n$-th coefficient of 
	\begin{equation*}
		\frac{R}{w-\log2}=-\frac{R}{\log2}\sum_n\Bigl(\frac{w}{\log2}\Bigr)^n
	\end{equation*}
	is $-R(\log2)^{-n-1}$. Therefore 
	\begin{equation*}
		\frac{\ell_n(x;\nu)}{n!}=-R(\log2)^{-n-1}(1+O((\log2/r)^n)),
	\end{equation*}
	which is \eqref{eq:asymp}. For $\nu=2$ the poles are at $\pm2\pi ik$, $k\ne0$%
	, and the stated formula for $\ell_n(2)$ follows from \eqref{eq:cor-stirling}%
	, with $B_{2m}\sim(-1)^{m-1}2(2m)!/(2\pi)^{2m}$ and $B_{2m+1}=0$.
\end{proof}

For $\nu=4$, $n=9$ formula \eqref{eq:asymp} gives $9!\,(\log2)^{-6}/4\approx
8.18\cdot10^{5}$ against the exact $807036$; for $\nu=3$, $n=9$ it gives $%
-9!\,(\log2)^{-7}/4\approx-1.18\cdot 10^{6}$ against $-1179864$. This also
explains why the sign of $\ell_n(\nu)$ is eventually constant for $\nu\ge3$
whereas it alternates for $\nu\le2$.

\section{Arbitrary nodes}

\label{sec:nodes}

The arguments of Sections \ref{sec:pf}--\ref{sec:asymp} use the specific
nodes $0,1,\dots,\nu-1$ only through the values of $A_j$ and of the sums $%
H^{(r)}_{\nu-2}$. We therefore state the general result. Let $\mathbf{a}%
=(a_0,\dots,a_{\nu-1})$ be a vector of distinct complex numbers and put 
\begin{equation}  \label{def-ell-a}
	\frac{w^{\nu}e^{xw}}{\prod_{j=0}^{\nu-1}(e^{w}-a_j)}=\sum_{n=0}^{\infty}%
	\ell_n(x;\mathbf{a})\frac{w^{n}}{n!},\qquad A_j(\mathbf{a}):=\prod 
	_{\substack{ i=0  \\ i\neq j}}^{\nu-1}\frac{1}{a_j-a_i}.
\end{equation}

\begin{theorem}
	\label{thm:nodes} \emph{(i)} If $1\notin\{a_0,\dots,a_{\nu-1}\}$, then $%
	\ell_n(x;\mathbf{a})=0$ for $n<\nu$ and, for $n\ge \nu$, 
	\begin{equation*}
		\ell_n(x;\mathbf{a})=-(n)_{\nu}\sum_{j=0}^{\nu-1}\frac{A_j(\mathbf{a})}{a_j-1%
		}F_{n-\nu}\!\left(x;\frac{1}{a_j-1}\right) =(n)_{\nu-1}\sum_{j=0}^{\nu-1}%
		\frac{A_j(\mathbf{a})}{a_j}\mathcal{B}_{n-\nu+1}\!\left(x;\frac1{a_j}\right),
	\end{equation*}
	the second form requiring in addition $0\notin\{a_j\}$.
	
	\emph{(ii)} If $a_{j_0}=1$ for one index $j_0$, then $\ell_n(x;\mathbf{a})=0$
	for $n<\nu-1$, $\ell_{\nu-1}(x;\mathbf{a})=(\nu-1)!A_{j_0}(\mathbf{a})$, and
	for $n\ge \nu$ the formula of {\rm (i)} holds with the $j_0$-th term
	replaced by $A_{j_0}(\mathbf{a})(n)_{\nu-1}B_{n-\nu+1}(x)$.
	
	\emph{(iii)} If moreover $|a_j|>1$ for all $j\ne j_0$, then with 
	\begin{equation*}
		Z_{\mathbf{a}}(s;x):=-A_{j_0}(\mathbf{a})\zeta(s,x)-\sum_{j\ne j_0}\frac{A_j(%
			\mathbf{a})}{a_j}\Phi\!\left(\frac1{a_j},s,x\right)\qquad(x>0)
	\end{equation*}
	one has $\ell_n(x;\mathbf{a})=(n)_{\nu}Z_{\mathbf{a}}(\nu-n;x)$ for $n\ge
	\nu $.
\end{theorem}

\begin{proof}
	Identical to the proofs of Theorems \ref{thm:pf} and \ref{thm:interp}, with
	Lemma \ref{lem:pf} replaced by the Lagrange partial fraction formula $%
	1/\prod_j(u-a_j)=\sum_j A_j(\mathbf{a})/(u-a_j)$. In (iii), $|1/a_j|<1$
	ensures that $\Phi(1/a_j,s,x)$ is entire in $s$ and that the expansion 
	\begin{equation*}
		\frac{e^{xw}}{e^w-a_j}=-\sum_r a_j^{-r-1}e^{(r+x)w}
	\end{equation*}
	is valid for $\operatorname{Re}w<\log|a_j|$.
\end{proof}

\begin{example}\label{ex:nodes}
	(a) $\nu=1$, $\mathbf{a}=(\lambda)$: $\ell_n(x;(\lambda))=\lambda^{-1}
	\mathcal{B}_n(x;1/\lambda)$, the Apostol--Bernoulli polynomials; for $\lambda=1$
	the Bernoulli polynomials; for $\lambda=2$, $\ell_n(x;(2))=-nF_{n-1}(x;1)$,
	so the (two--variable) Fubini polynomials also belong to the family.
	
	(b) $\mathbf{a}=(1,-1)$: $w^{2}e^{xw}/(e^{2w}-1)=(w/2)\cdot(2w)e^{(x/2)(2w)}/(e^{2w}-1)$, so $\ell_n(x;(1,-1))=n\,2^{n-2}B_{n-1}(x/2)$;
	Theorem \ref{thm:nodes}(ii) gives the same thing in the form 
	\[\frac12(n)_{1}B_{n-1}(x)-(n)_{2}\cdot\frac{-1/2}{-2}F_{n-2}(x;-\frac12),\]
	and since 
	\[\frac{e^{xw}}{1+\frac12(e^w-1)}=\frac{2e^{xw}}{e^w+1},\]
	i.e., $F_n(x;-1/2)=E_n(x)$ (Euler polynomials), this is the classical relation 
	\begin{equation*}
		2^{n-2}B_{n-1}(x/2)=\frac12B_{n-1}(x)-\frac{n-1}{4}E_{n-2}(x).
	\end{equation*}
	
	(c) The numbers $Y_n(k,a)$ and polynomials $Q_n(x,k,a)$ of \cite{Simsek2021b}
	(see also \cite{SimsekSurvey}) are generated by $a\,t\,e^{(k+1+x)t}/\big(%
	(e^{(k+2)t}-1)(e^{kt}+1)\big)$. Since 
	\begin{equation*}
		e^{(k+2)t}-1=\prod_{\xi^{k+2}=1}(e^{t}-\xi)
	\end{equation*}
	and $e^{kt}+1=\prod_{\eta^{k}=-1}(e^{t}-\eta)$, the denominator is $%
	\prod(e^{t}-a_j)$ over the $\nu=2k+2$ roots $a_j$ of $u^{k+2}=1$ and $%
	u^{k}=-1$; these are pairwise distinct precisely when $k\not\equiv2\pmod4$
	(a common root would satisfy $u^2=-1$, forcing $4\mid k+2$). In that case
	exactly one node equals $1$, and 
	\begin{equation*}
		Q_n(x,k,a)=a\,\frac{n!}{(n+2k+1)!}\,\ell_{n+2k+1}\big(k+1+x;\mathbf{a}\big),
	\end{equation*}
	because dividing the generating function \eqref{def-ell-a} by $t^{\nu-1}$
	shifts the index by $\nu-1=2k+1$ and $\ell_m(\cdot;\mathbf{a})$ vanishes for 
	$m<\nu-1$ by (ii). Theorem \ref{thm:nodes} then expresses $Q_n(x,k,a)$
	through Bernoulli polynomials and two--variable Fubini polynomials at the
	complex arguments $1/(\xi-1)$, and (iii) is not applicable since $|a_j|=1$;
	the interpolating function is instead a combination of Lerch functions $%
	\Phi(\bar\xi,s,x)$ at roots of unity, i.e., of the kind that appears in the
	theory of twisted Bernoulli numbers and Dedekind sums \cite{SimsekSurvey}.
	We checked the displayed identity for $k=1$ (nodes $1,e^{\pm2\pi i/3},-1$)
	up to $n=5$; a general treatment of the case $k\equiv2\pmod 4$ (double node)
	would require partial fractions with a double pole and is not carried out
	here.
\end{example}

\section{Conclusion}

1. The two expansions of $1/(u)_{\nu}$ used here, at $u=\infty$ (partial
fractions, Section \ref{sec:pf}) and at $u=1$ (Section \ref{sec:harmonic}),
are the two natural ones. An expansion at $u=0$, i.e., in powers of $e^{w}$,
produces the complete homogeneous symmetric functions $%
h_m(1,\tfrac12,\dots,\tfrac1{\nu-1})$; as a series in $w$ around $w=0$ it
does not converge (its terms do not tend to zero), but for $w=-t<0$ it does,
and under the Mellin transform it is precisely the series \eqref{DM-1} of %
\cite{Peneva1} (Proposition \ref{prop:PRS}).

2. The recursion \eqref{eq:raise} and the harmonic form \eqref{eq:taylor}
both single out the node $j=1$; it would be interesting to know whether the $%
2$--adic (or $p$--adic) properties of $\ell_n(\nu)$, which are visible in
Table \ref{tab:values} (for $n\le9$ all $\ell_n(4)$ are integers and all $%
\ell_n(3)$ have denominator $1$ or $3$; for larger $n$ further primes enter
through the Bernoulli term of \eqref{eq:cor-stirling}, e.g.\ $\ell_{14}(3)$
has denominator $15$ and $\ell_{19}(4)$ has denominator $5$) can be
explained by a $p$--adic integral representation of the pieces $w/(e^{w}-j)$,
$j\ge2$, in the spirit of the $p$--adic integral methods discussed in \cite%
{SimsekSurvey}. The $j=1$ piece is the Volkenborn integral of $e^{xw}$, and
the pieces with $j\equiv1\pmod p$ can be treated by the $\lambda$--twisted
integral with $\lambda=1/j$, but the remaining values of $j$ do not seem to
be accessible in this way.

3. It is natural to ask for a $q$--analogue, replacing the nodes $j$ by $%
[j]_q$. Theorem \ref{thm:nodes} applies verbatim (the nodes $[j]_q$ are
distinct for generic $q$), so the only real question is what replaces the
harmonic numbers in Theorem \ref{thm:taylor}; presumably the $q$--harmonic
numbers $\sum_{i}1/[i]_q^{r}$, but this has not been verified.

4. We do not know a combinatorial interpretation of the (integer) numbers $%
\ell_n(4)$, nor of $(\nu-1)!\,\ell_n(\nu)$ in general. Proposition \ref%
{prop:moments} shows that they are signed combinations of moments of
geometric variables, which suggests looking for a signed enumeration of
sequences of Bernoulli trials, but we have no definite result in this
direction.
\smallskip

{\bf Acknowledgments.}
The work of the second author was supported in part by the Serbian Academy of Sciences and Arts ($\Phi$-96).

\end{document}